\documentclass[11pt]{article}
\usepackage{amsmath,amsthm}
\usepackage{amssymb,mathrsfs}
\usepackage{bm,mathtools}
\usepackage{graphics}
\usepackage{tikz}
\usepackage{float}
\usepackage{enumitem}
\usepackage{bm}
\usetikzlibrary{calc}

\usepackage{xcolor}

\usepackage{geometry}
\usepackage[colorlinks=true,
linkcolor=blue,citecolor=blue,pagebackref,
urlcolor=blue]{hyperref}

\makeatletter
\def\@seccntDot{.}
\def\@seccntformat#1{\csname the#1\endcsname\@seccntDot\hskip 0.5em}
\renewcommand\section{\@startsection{section}{1}{\z@}%
	{18\p@ \@plus 6\p@ \@minus 3\p@}%
	{9\p@ \@plus 6\p@ \@minus 3\p@}%
	{\large\bfseries\boldmath}}
\renewcommand\subsection{\@startsection{subsection}{2}{\z@}%
	{12\p@ \@plus 6\p@ \@minus 3\p@}%
	{3\p@ \@plus 6\p@ \@minus 3\p@}%
	{\bfseries\boldmath}}
\renewcommand\subsubsection{\@startsection{subsubsection}{3}{\z@}%
	{12\p@ \@plus 6\p@ \@minus 3\p@}%
	{\p@}%
	{\bfseries\boldmath}}
\makeatother

\usepackage{microtype}
\usepackage{float}

\newcommand{\keywords}[1]{%
  \par\medskip
  \noindent{\small\textbf{Keywords:} #1\par}
  \medskip
}

\newcommand{\msc}[1]{%
  \par\smallskip
  \noindent{\small\textbf{MSC:} #1\par}
  \medskip
}

\theoremstyle{plain}
\newtheorem{theorem}{Theorem}[section]
\newtheorem{lemma}[theorem]{Lemma}

\theoremstyle{definition}

\newtheorem{claim}{Claim}[section]

\numberwithin{equation}{section}
\allowdisplaybreaks
\newcommand{\floor}[1]{\left\lfloor #1\right\rfloor}
\newcommand{\ceil}[1]{\left\lceil #1\right\rceil}
\newcommand{\T}{\mathsf{T}}

\title{The Signless Laplacian Spectral Radius of $tK_3$-Free Graphs}

\author{
	Jing Zeng\thanks{College of Cryptology and Cyber Science, Nankai University, Tianjin 300350, P.R. China. E-mail: \texttt{jingzeng@mail.nankai.edu.cn}. }
}

\date{}

\begin{document}
	\maketitle		
	
	\begin{abstract}
		The signless Laplacian matrix of a graph $G$ is $Q(G)=D(G)+A(G)$, where $D(G)$ and $A(G)$ are the diagonal degree matrix and the adjacency matrix of $G$, respectively. The signless Laplacian spectral radius of $G$ is the largest eigenvalue of $Q(G)$.  For a positive integer $t$, a graph is called $tK_3$-free if it contains no $t$ vertex-disjoint triangles. In this paper, for every fixed $t\geq 2$ and all $n\geq 28t-17$, we determine the unique graph achieving the maximum signless Laplacian spectral radius among all $tK_3$-free graphs of order $n$.
	\end{abstract}

    \msc{05C35, 05C50}

\section{Introduction}\label{sec-intro}

All graphs considered in this paper are finite, undirected and simple. The \textit{adjacency matrix} and \textit{diagonal degree matrix} of $G$ are denoted by $A(G)$ and $D(G)$, respectively. The largest eigenvalue of $A(G)$ is commonly known as the \textit{(adjacency) spectral radius} of $G$. The \textit{signless Laplacian matrix} of $G$ is defined as $Q(G)=D(G)+A(G)$, and its largest eigenvalue is denoted by $q(G)$ and called the \textit{signless Laplacian spectral radius} of $G$.

Let $F$ be a graph. A graph $G$ is \textit{$F$-free} if it does not contain a subgraph isomorphic to $F$, and $G$ is \textit{$tF$-free} if it does not contain $t$ pairwise vertex-disjoint copies of $F$ as subgraphs.

Mantel's theorem \cite{Mantel} for triangle-free graphs and Tur\'{a}n's theorem \cite{Turan} for $K_{r+1}$-free graphs are two foundational results in extremal graph theory. A general Tur\'{a}n-type problem is to determine the maximum number of edges in an $F$-free graph of order $n$.
In 2010, Nikiforov \cite{Nikispecturan} proposed a spectral analogue of the Tur\'{a}n-type problem, namely, to determine the maximum spectral radius of an $F$-free graph of order $n$. For the adjacency spectral radius, there is a wealth of results on various forbidden subgraphs, such as the complete graph $K_{r+1}$ \cite{Niki2007}, the complete bipartite graphs \cite{BabaiGuiduli,NikiZaraprob}, the cycles of given lengths \cite{ZhaiLinC6,ZhaiWangC4}, the wheels \cite{CioabaDTWheel,ZhaoHLWheel}, and so on.
For the signless Laplacian spectral version of the Tur\'{a}n-type problem, there has also been extensive research on some special graphs, such as the complete graph $K_{r+1}$ \cite{HeJZKr}, the odd cycles \cite{YuanoddC}, the even cycles \cite{NikiYuanevenC}, the union of $t$ vertex-disjoint copies of $K_2$ \cite{YuK2}, and so on. More recent results on both types of spectral Tur\'{a}n-type problems can be found in \cite{LiuNingSparse,LiuMiaobook,ZhengLiSu26,ZhengLF26}. For open problems in spectral graph theory, the reader is referred to \cite{LiuNingUnsolve}.

In 2023, for $t \geq 2$ and $r \geq 2$, Ni, Wang, and Kang \cite{NiWangKang23} determined the unique extremal graph with maximum adjacency spectral radius among all $tK_{r+1}$-free graphs of sufficiently large order $n$. It is natural to consider the corresponding problem for the signless Laplacian spectral radius. When $r=2$, the fact that a graph is $tK_{r+1}$-free implies that it contains no $t$ vertex-disjoint triangles. In 2024, Zhang and Wang \cite{ZhangWang2K3} determined the unique extremal graph with maximum signless Laplacian spectral radius among $2K_3$-free graphs of order $n\geq 44$. This case was later completed for all $n\geq 6$ by Zhang and Lei \cite{ZhangLei2K3}. Moreover, Zhang and Wang \cite{ZhangWang3K3} determined the extremal $3K_3$-free graphs for $n\geq 453$.

In this paper, we determine the unique extremal graph attaining the maximum signless Laplacian spectral radius among all $tK_3$-free graphs of order $n$ for each $t\geq 2$ and all sufficiently large $n$. The result is stated as follows.

\begin{theorem}\label{mainthm}
Let $t\geq 2$, and let $G$ be a $tK_3$-free graph of order $n\geq 28t-17$ with maximum signless Laplacian spectral radius. Then $G= K_{t-1}\vee K_{\floor{(n-t+1)/2},\ceil{(n-t+1)/2}}$.
\end{theorem}

Throughout this paper, for a graph $G$, we write $V(G)$ and $E(G)$ for its vertex set and edge set, respectively, and write $e(G)=|E(G)|$. For $v\in V(G)$, let $d_G(v)$ and $N_G(v)$ denote its \textit{degree} and \textit{neighborhood}. Let $X$ and $Y$ be disjoint subsets of $V(G)$. We use $E(X,Y)$ to denote the set of edges between $X$ and $Y$, and $E(X)$ to denote the set of edges with both endpoints in $X$. Write $e(X,Y)=|E(X,Y)|$ and $e(X)=|E(X)|$. For two vertex-disjoint graphs $G$ and $H$, their \textit{union} is denoted by $G \cup H$, and their \textit{join} is denoted by $G\vee H$.  For a graph $F$ and a positive integer $t$, we use $tF$ to denote the union of $t$ vertex-disjoint copies of $F$. As usual, $K_n$ and $K_{m,n}$ denote the \textit{complete graph} of order $n$ and the \textit{complete bipartite graph} with partite sets of sizes $m$ and $n$, respectively. For other undefined notation and terminology, the reader is referred to \cite{BondyM}.

The remainder of this paper is organized as follows. In Section \ref{sec-pre}, we introduce some results that will be used in the proof. In Section \ref{sec-pf}, we prove Theorem \ref{mainthm}.

\section{Preliminaries}\label{sec-pre}

Let $M$ be a real symmetric matrix of order $n$, and let $V = \{1,2,\dots,n\}$. Given a partition $\Pi : V = V_1\cup\cdots\cup V_k$, then, according to the partition $\Pi$, the matrix $M$ can be written as
\[
M =\begin{bmatrix}
M_{1,1} & M_{1,2} & \cdots & M_{1,k} \\
M_{2,1} & M_{2,2} & \cdots & M_{2,k} \\
\vdots & \vdots & \ddots & \vdots \\
M_{k,1} & M_{k,2} & \cdots & M_{k,k}
\end{bmatrix}.
\]
The \textit{quotient matrix} of $M$ with respect to the partition $\Pi$ is the $k\times k$ matrix $B_\Pi = (b_{i,j})$, where each entry $b_{i,j}$ is the average of the row sums of the block $M_{i,j}$. If for all $i,j \in \{1,2,\dots,k\}$, all row sums of $M_{i,j}$ are equal to $b_{i,j}$, then $\Pi$ is called an \textit{equitable partition} of $M$, and the matrix $B_\Pi = (b_{i,j})_{k\times k}$ is called an \textit{equitable quotient matrix} of $M$.

\begin{lemma}[p. 30, \cite{BrouwerH}; p. 196--198, \cite{GodsilR}]\label{lemquotient}
Let $M$ be a real symmetric matrix and let $B_{\Pi}$ be an equitable quotient matrix of $M$. Then the eigenvalues of $B_{\Pi}$ are also eigenvalues of $M$. Furthermore, if $M$ is nonnegative and irreducible, then $\lambda(M) = \lambda(B_{\Pi})$, where $\lambda(M)$ and $\lambda(B_{\Pi})$ are the largest eigenvalues of $M$ and $B_{\Pi}$, respectively. 
\end{lemma}

\begin{lemma}[\cite{DasLap,FengYuLap}]\label{lemlocalbound}
Let $G$ be a connected graph.  Then
\begin{equation*}
q(G)\leq \max_{u\in V(G)}\left\{d_G(u)+\frac{1}{d_G(u)}\sum_{v\in N_G(u)}d_G(v)\right\}.
\end{equation*}
\end{lemma}

\begin{lemma}[\cite{CvetkovicRS}]\label{lemmonotone}
Let $H$ be a proper subgraph of a connected graph $G$.  Then $q(H)<q(G)$.
\end{lemma}

\begin{lemma}\label{lemjoin}
Let $t> r\geq 1$.  If $H$ is a $(t-r)K_3$-free graph, then $K_r\vee H$ is $tK_3$-free.
\end{lemma}

\begin{proof}
Suppose that $K_r\vee H$ contains $t$ vertex-disjoint triangles. Since $K_r$ has exactly $r$ vertices, there are at most $r$ triangles among these $t$ vertex-disjoint triangles that contain a vertex of $K_r$. Hence at least $t-r$ triangles of the $t$ vertex-disjoint triangles are contained entirely in $H$, which contradicts the assumption that $H$ is $(t-r)K_3$-free.
\end{proof}

Recall that a \textit{vertex cover} of a graph is a set of vertices that meets every edge.
\begin{lemma}\label{lemcover}
Let $\ell\geq 1$, and let $G$ be a graph of order $n$. If $G$ has a vertex cover of size at most $\ell$, then
\[
e(G)\leq
\begin{cases}
\displaystyle \binom{n}{2}, & n\leq \ell+1,\\[2mm]
\displaystyle \ell n-\frac{\ell(\ell+1)}{2}, & n> \ell+1.
\end{cases}
\]
\end{lemma}

\begin{proof}
Let $X$ be a vertex cover of $G$ with $|X|\leq \ell$. Since $X$ is a vertex cover of $G$, every edge of $G$ is incident with a vertex in $X$. If $uv$ is an edge of $G$, then we have $uv \in E(X)$ or $uv \in E(X,V(G)\setminus X)$. Thus we have
\[
e(G)\leq \binom{|X|}{2}+|X|(n-|X|).
\]
For $n\leq \ell+1$, it is trivial that $e(G)\leq \binom{n}{2}$. For $n> \ell+1$, $\binom{|X|}{2}+|X|(n-|X|)$ attains its maximum when $|X|=\ell$, and thus $e(G)\leq \ell n-\frac{\ell(\ell+1)}{2}$. This completes the proof.
\end{proof}

\section{Proof of Theorem \ref{mainthm}}\label{sec-pf}

Throughout this section, for an integer $t\geq 1$ and all $x > 1$, we define
\begin{equation}\label{eqdefR}
R_t(x)=6t-\frac{3t(3t+1)}{x-1}.
\end{equation}
For positive integers $m$ and $n$, we use $J_{m\times n}$ to denote the $m\times n$ all-ones matrix, and use $\bm{1}_m$ to denote the all-ones vector of size $m$.

Before giving the proof of Theorem \ref{mainthm}, we first establish two important lemmas.

\begin{lemma}\label{lem1forthm}
Let $t\geq 1$, and let $G$ be a $(t+1)K_3$-free graph of order $n\geq 3t+2$.  If $q(G)>n+R_{t}(n)$, then there exists a vertex $v\in V(G)$ such that $G-\{v\}$ is $tK_3$-free.
\end{lemma}

\begin{proof}
Clearly, there exists a connected component $C$ of $G$ with the same signless Laplacian spectral radius, that is, $q(C)=q(G)$. Let $m=|V(C)|$. If $m\leq 3t+1$, then all row sums of $Q(C)$ are at most $2m-2\leq 6t$, and hence $q(C)\leq 6t$.  On the other hand, it is easy to check that $g_t(x)=x+R_t(x)$ is strictly increasing for $x> 1$ and $g_t(3t+2)=6t+2$. Then by $q(G)>n+R_{t}(n)$, we have $q(C)=q(G)>g_t(n)\geq 6t+2$, a contradiction. Thus $m \geq 3t+2$. Since $g_t(x)$ is strictly increasing and $n\geq m$, we have
\begin{equation}\label{eqlem1qdege}
q(C)>m+R_t(m).
\end{equation}
Since $C$ is connected and $m\geq 3t+2\geq 5$, every vertex of $C$ has positive degree. Choose $u\in V(C)$ such that $d_C(u)+\frac{1}{d_C(u)}\sum_{v\in N_C(u)}d_C(v)$ is maximum.  By Lemma \ref{lemlocalbound},
\begin{equation}\label{eqlem1local}
q(C)\leq d_C(u)+\frac{1}{d_C(u)}\sum_{v\in N_C(u)}d_C(v).
\end{equation}
Let $W=V(C)\setminus (N_C(u)\cup\{u\})$. By \eqref{eqlem1local}, we have
\begin{align}
q(C)&\leq d_C(u)+1+\frac{2e(C[N_C(u)])+e(N_C(u),W)}{d_C(u)} \notag\\
&\leq d_C(u)+1+\frac{2e(C[N_C(u)])+d_C(u)(m-1-d_C(u))}{d_C(u)} \notag\\
&=m+\frac{2e(C[N_C(u)])}{d_C(u)}. \label{eqlem1qleq1}
\end{align}

We now prove that $G-\{u\}$ is $tK_3$-free.  Suppose to the contrary that $G-\{u\}$ contains $t$ vertex-disjoint triangles. Let $S$ be the union of the vertex sets of $t$ vertex-disjoint triangles in $G-\{u\}$.  Then $|S|=3t$.  Since $G$ is $(t+1)K_3$-free, if there is an edge $v_1v_2 \in E(C[N_C(u)])$ such that $v_1,v_2\notin S$, then $\{u,v_1,v_2\}$ would form one more triangle disjoint from the $t$ triangles in $S$, a contradiction. Thus $S \cap N_C(u)$ is a vertex cover of $C[N_C(u)]$ with size at most $\ell=3t$. By Lemma \ref{lemcover}, if $d_C(u)\leq \ell+1$, then
\[
\frac{2e(C[N_C(u)])}{d_C(u)}\leq d_C(u)-1\leq \ell\leq R_t(m),
\]
where the last inequality follows from $m\geq 3t+2 = \ell+2$.  If $d_C(u)\geq \ell+2$, then
\[
e(C[N_C(u)])\leq \ell d_C(u)-\frac{\ell(\ell+1)}{2},
\]
which together with $d_C(u)\leq m-1$ yields
\[
\frac{2e(C[N_C(u)])}{d_C(u)}\leq 2\ell-\frac{\ell(\ell+1)}{d_C(u)}\leq 2\ell-\frac{\ell(\ell+1)}{m-1}=R_t(m).
\]
In both cases, using \eqref{eqlem1qleq1} gives $q(C)\leq m+R_t(m)$, which contradicts \eqref{eqlem1qdege}.  Therefore $G-\{u\}$ is $tK_3$-free. This completes the proof.
\end{proof}

\begin{lemma}\label{lem2forthm}
Let $s\geq 1$ and $G=K_s\vee H$, where $H$ is a $K_3$-free graph with $|V(H)|\geq 1$. Let $\bm{x}=(x_v)_{v\in V(G)}$ be a positive unit Perron vector of $Q(G)$, and choose $w\in V(H)$ such that $x_w=\max\{x_v:v\in V(H)\}$. Let $H'$ be the complete bipartite graph with parts $N_H(w)$ and $V(H)\setminus N_H(w)$, and let $G'=K_s\vee H'$. Then $q(G')\geq q(G)$. Moreover, if $N_H(w)\neq\varnothing$ and $H\neq H'$, then $q(G')>q(G)$.
\end{lemma}

\begin{proof}
For each $v\in V(G)$, let
\[
S_G(v)=\sum_{uv\in E(G)}(x_u+x_v),\qquad
S_{G'}(v)=\sum_{uv\in E(G')}(x_u+x_v).
\]
Since $\bm{x}$ is the Perron vector of $Q(G)$, we have $Q(G)\bm{x}=q(G)\bm{x}$. Then for all $v\in V(G)$,
\begin{equation}\label{eqlem2Seigen}
S_G(v)=q(G)x_v.
\end{equation}
In particular, $q(G)x_w=d_G(w)x_w+\sum_{u\in N_G(w)}x_u$. Since $x_v>0$ for all $v\in V(G)$ and $N_{G}(w)\neq \varnothing$, we have $q(G)>d_G(w)$.

For every $v^{(1)}\in N_H(w)$, since $H$ is $K_3$-free, $E(N_H(w))=\varnothing$. Hence the graph $G'$ only adds possible new neighbors of $v^{(1)}$ in $V(H)\setminus (N_H(w)\cup \{w\})$. It follows that $N_G(v^{(1)}) \subseteq N_{G'}(v^{(1)})$, and thus
\begin{equation}\label{eqlem2v1}
S_{G'}(v^{(1)})\geq S_G(v^{(1)}).
\end{equation}
For every $v^{(2)}\in V(K_s)\cup\{w\}$, we have $N_G(v^{(2)}) = N_{G'}(v^{(2)})$. Thus
\begin{equation}\label{eqlem2v2}
S_{G'}(v^{(2)})=S_G(v^{(2)}).
\end{equation}
For every $v^{(3)}\in V(H)\setminus (N_H(w)\cup \{w\})$, we have $N_{G'}(v^{(3)})=V(K_s)\cup N_H(w)$.  Thus
\begin{align*}
S_{G'}(v^{(3)})&=d_G(w)x_{v^{(3)}}+S_G(w)-d_G(w)x_w \\
&=S_G(v^{(3)})+(q(G)-d_G(w))(x_w-x_{v^{(3)}}),
\end{align*}
where the last equality follows from \eqref{eqlem2Seigen}.  Since $x_w\geq x_{v^{(3)}}$ and $q(G)>d_G(w)$, we have
\begin{equation}\label{eqlem2v3}
S_{G'}(v^{(3)})\geq S_G(v^{(3)}).
\end{equation}
Combining \eqref{eqlem2v1}, \eqref{eqlem2v2} and \eqref{eqlem2v3} gives
\begin{equation*}
q(G')\geq \bm{x}^{\T}Q(G')\bm{x}
=\sum_{v\in V(G)}x_vS_{G'}(v)
\geq \sum_{v\in V(G)}x_vS_G(v)=q(G).
\end{equation*}

Now we assume that $N_H(w)\neq \varnothing$ and $H\neq H'$. If $G'$ contains an edge between $N_H(w)$ and $V(H)\setminus (N_H(w)\cup \{w\})$ that is not in $G$, then $S_{G'}(v^{(1)})> S_G(v^{(1)})$ for at least one vertex $v^{(1)}\in N_H(w)$, and we have $q(G')>q(G)$.  If $G'$ contains no such edge, then $H$ contains all edges between $N_H(w)$ and $V(H)\setminus (N_H(w)\cup \{w\})$. By $H\neq H'$, $H$ has an edge $v_1v_2 \in E\left(V(H)\setminus (N_H(w)\cup \{w\})\right)$.  Since $N_H(w)\neq \varnothing$, there exists a vertex $u\in N_H(w)$ adjacent to $v_1$ and $v_2$, which implies that $\{u,v_1,v_2\}$ forms a triangle in $H$, a contradiction. Therefore, if $N_H(w)\neq \varnothing$ and $H\ne H'$, then $q(G')>q(G)$.
\end{proof}

We now prove Theorem \ref{mainthm}.

\begin{proof}[{\bf Proof of Theorem \ref{mainthm}}]
For an integer $s\geq 2$ and all $x>s$, let \begin{equation}\label{eqthmdefLsn}
L_s(x)=x+\frac{s-2}{2}+\sqrt{s(x-s)+\frac{(s-2)^2}{4}}-\frac{1}{x-s}.
\end{equation}
For integers $s\geq 2$ and $1\leq h\leq s-1$, and all $x\geq s+2h+2$, let
\begin{align}\label{eqthmdefU}
    U_{s,h}(x)&=x+\frac{s-h-2+R_h(x-s+h)}{2} \notag\\
&\quad +\sqrt{(s-h)(x-s+h)+\frac{(s-h-2-R_h(x-s+h))^2}{4}}.
\end{align}
To complete the proof, we first show the following two claims.

\begin{claim}\label{clm1}
Let $s\geq 1$, $n\geq 28s+11$. Then $q\left(K_s\vee K_{\floor{(n-s)/2},\ceil{(n-s)/2}}\right)>n+R_s(n)$. Moreover, if $s\geq 2$, then $q\left(K_s\vee K_{\floor{(n-s)/2},\ceil{(n-s)/2}}\right)\geq L_s(n)>n+R_s(n)$.
\end{claim}

\begin{proof}
First suppose $s=1$. Let $A_1$ and $B_1$ be the two partite sets of $K_{\floor{(n-1)/2},\ceil{(n-1)/2}}$, where $|A_1|=\floor{\frac{n-1}{2}}$ and $|B_1|=\ceil{\frac{n-1}{2}}$. Let $a=|A_1|$ and $b=|B_1|$. The equitable quotient matrix of $Q\left(K_1\vee K_{\floor{(n-1)/2},\ceil{(n-1)/2}}\right)$ with respect to the partition $V(K_1),A_1,B_1$ is
\[
\begin{bmatrix}
n-1 & a & b\\
1 & b+1 & b\\
1 & a & a+1
\end{bmatrix},
\]
and its characteristic polynomial is $p(x)=x(x-n)^2-4ab$.

Since $p'(x)=(x-n)(3x-n)>0$ for $x>n$ and $p(n)=-4ab<0$, the polynomial $p(x)$ has a unique root in $(n,+\infty)$, namely $q(K_1\vee K_{a,b})$ by Lemma \ref{lemquotient}. Since $R_1(n)=6-\frac{12}{n-1}>0$, it suffices to prove that $p(n+R_1(n))<0$. If $n$ is odd, then $4ab=(n-1)^2$, and we have
\begin{equation*}
    4ab-(n+R_1(n))R^{2}_1(n)=\frac{(n^2+4n-17)(n^3-45n^2+243n-343)}{(n-1)^3}>0
\end{equation*}
for every odd $n\geq 39$. If $n$ is even, then $4ab=(n-1)^2-1$, and we have
\begin{equation*}
    4ab-(n+R_1(n))R^{2}_1(n)=\frac{n^5-41n^4+45n^3+1397n^2-5506n+5832}{(n-1)^3}>0
\end{equation*}
for every even $n\geq 40$. Thus $q\left(K_1\vee K_{\floor{(n-1)/2},\ceil{(n-1)/2}}\right)>n+R_1(n)$ for $n\geq 28s+11=39$.

Now assume $s\geq 2$. Let $A_s$ and $B_s$ be the two partite sets of $K_{\floor{(n-s)/2},\ceil{(n-s)/2}}$, where $|A_s|=\floor{\frac{n-s}{2}}$ and $|B_s|=\ceil{\frac{n-s}{2}}$. Let $a=|A_s|$ and $b=|B_s|$. For simplicity, let $Q=Q\left(K_s\vee K_{\floor{(n-s)/2},\ceil{(n-s)/2}}\right)$ with the rows and columns ordered according to the vertex partition $V(K_s),A_s,B_s$. Let $V=V\left(K_s\vee K_{\floor{(n-s)/2},\ceil{(n-s)/2}}\right)$, and let $\bm{\alpha}=(\alpha_v)_{v\in V}$ and $\bm{\beta}=(\beta_v)_{v\in V}$ such that
\[
\alpha_v=\begin{cases}
    1, & \text{if $v\in V(K_s)$};\\
    0, & \text{if $v\notin V(K_s)$},
\end{cases}\quad\quad
\beta_v=\begin{cases}
    1, & \text{if $v\in V(K_{\floor{(n-s)/2},\ceil{(n-s)/2}})$};\\
    0, & \text{if $v\notin V(K_{\floor{(n-s)/2},\ceil{(n-s)/2}})$}.
\end{cases}
\]
Then we have
\begin{equation}\label{eqclm1one}
    \begin{bmatrix}
    \frac{1}{\sqrt{s}}\bm{\alpha}^{\T}\\
    \frac{1}{\sqrt{n-s}}\bm{\beta}^{\T}
\end{bmatrix}Q\begin{bmatrix}
    \frac{1}{\sqrt{s}}\bm{\alpha} & \frac{1}{\sqrt{n-s}}\bm{\beta}
\end{bmatrix}=\begin{bmatrix}
    \frac{1}{s}\bm{\alpha}^{\T} Q \bm{\alpha} & \frac{1}{\sqrt{s(n-s)}}\bm{\alpha}^{\T} Q \bm{\beta}\\
    \frac{1}{\sqrt{s(n-s)}}\bm{\beta}^{\T} Q \bm{\alpha} & \frac{1}{n-s}\bm{\beta}^{\T} Q \bm{\beta}
\end{bmatrix}.
\end{equation}
Note that
\begin{equation*}
Q=\begin{bmatrix}
(n-1)I_s+A(K_s) & J_{s\times a} & J_{s\times b}\\
J_{a\times s} & (s+b)I_a & J_{a\times b}\\
J_{b\times s} & J_{b\times a} & (s+a)I_b
\end{bmatrix}.    
\end{equation*}
Combining this with \eqref{eqclm1one} gives
\begin{equation}\label{eqclm1two}
    \begin{bmatrix}
    \frac{1}{\sqrt{s}}\bm{\alpha}^{\T}\\
    \frac{1}{\sqrt{n-s}}\bm{\beta}^{\T}
\end{bmatrix}Q\begin{bmatrix}
    \frac{1}{\sqrt{s}}\bm{\alpha} & \frac{1}{\sqrt{n-s}}\bm{\beta}
\end{bmatrix}=\begin{bmatrix}
 n+s-2 & \sqrt{s(n-s)}\\
 \sqrt{s(n-s)} & s+\frac{4ab}{n-s}
\end{bmatrix}.
\end{equation}
Clearly, 
\begin{equation*}
    \begin{bmatrix}
    \frac{1}{\sqrt{s}}\bm{\alpha}^{\T}\\
    \frac{1}{\sqrt{n-s}}\bm{\beta}^{\T}
\end{bmatrix}\begin{bmatrix}
    \frac{1}{\sqrt{s}}\bm{\alpha} & \frac{1}{\sqrt{n-s}}\bm{\beta}
\end{bmatrix}=\begin{bmatrix}1 & 0\\0 & 1\end{bmatrix}.
\end{equation*}
Since $ab\geq \frac{((n-s)^2-1)}{4}$, we have $s+\frac{4ab}{n-s}\geq n-\frac{1}{n-s}$. It follows from \eqref{eqclm1two} and Rayleigh's principle that
\begin{align*}
q\left(K_s\vee K_{\floor{(n-s)/2},\ceil{(n-s)/2}}\right)&=\max_{\|\bm{y}\|=1}\bm{y}^{\T}Q\bm{y}\\
&\geq \max_{\substack{\bm{z}\in\mathbb{R}^{2}\\ \|\bm{z}\|=1}}\bm{z}^{\T}\begin{bmatrix}
    \frac{1}{\sqrt{s}}\bm{\alpha}^{\T}\\
    \frac{1}{\sqrt{n-s}}\bm{\beta}^{\T}
\end{bmatrix}Q\begin{bmatrix}
    \frac{1}{\sqrt{s}}\bm{\alpha} & \frac{1}{\sqrt{n-s}}\bm{\beta}
\end{bmatrix}\bm{z}\\
&= \lambda_{\max}
\begin{pmatrix}
 n+s-2 & \sqrt{s(n-s)}\\
 \sqrt{s(n-s)} & s+\frac{4ab}{n-s}
\end{pmatrix} \\&\geq \lambda_{\max}
\begin{pmatrix}
 n+s-2 & \sqrt{s(n-s)}\\
 \sqrt{s(n-s)} & n-\frac{1}{n-s}
\end{pmatrix} \\
&=n+\frac{s-2}{2}-\frac{1}{2(n-s)}
+\sqrt{s(n-s)+\frac{(s-2+\frac{1}{n-s})^2}{4}}\\
&\geq n+\frac{s-2}{2}+\sqrt{s(n-s)+\frac{(s-2)^2}{4}}-\frac{1}{n-s}=L_s(n).
\end{align*}
It remains to show $L_s(n)>n+R_s(n)$ for $n\geq 28s+11$.  Let $f_s(x)=L_s(x)-x-R_s(x)$. For $x\geq 28s+11$, by \eqref{eqdefR} and \eqref{eqthmdefLsn},  we have
\[
f_s'(x)=\frac{s}{2\sqrt{s(x-s)+\frac{(s-2)^2}{4}}}+\frac{1}{(x-s)^2}-\frac{3s(3s+1)}{(x-1)^2}.
\]
Note that $s(x-s)+\frac{(s-2)^2}{4}=sx-\frac{3s^2+4s-4}{4}
<sx<x^2$. Hence $\sqrt{s(x-s)+\frac{(s-2)^2}{4}}<x$. Moreover, $\frac{(x-1)^2}{x}
=x-2+\frac{1}{x} \geq28s+9+\frac{1}{x} >18s+6 =6(3s+1)$. It follows that $\frac{s}{2\sqrt{s(x-s)+\frac{(s-2)^2}{4}}}>\frac{s}{2x}>\frac{3s(3s+1)}{(x-1)^2}$. Therefore, $f_s'(x)>\frac{1}{(x-s)^2}>0$.

Now it suffices to check $f_s(28s+11)>0$. If $s=2$, then $f_2(67)=\sqrt{130}-\frac{8136}{715}>0$. If $s\geq 3$, let $n_0=28s+11$. Then $s(n_0-s)+\frac{(s-2)^2}{4}=\frac{109s^2}{4}+10s+1$. Moreover, 
\begin{equation*}
s(n_0-s)+\frac{(s-2)^2}{4}= s(27s+11)+\frac{(s-2)^2}{4}=\frac{109}{4}s^2+10s+1>\frac{676}{25}s^2.
\end{equation*}
It follows that $\sqrt{s(n_0-s)+\frac{(s-2)^2}{4}}>\frac{26s}{5}$. Then we obtain
\begin{align*}
\sqrt{s(n_0-s)+\frac{(s-2)^2}{4}}-\frac{11s+2}{2}
&=-\frac{3s^2+s}{\frac{11s+2}{2}+\sqrt{s(n_0-s)+\frac{(s-2)^2}{4}}}\\
&>-\frac{10s(3s+1)}{107s+10}.
\end{align*}
Consequently, for $s\geq 3$, we have
\begin{align*}
f_s(n_0)&=\sqrt{s(n_0-s)+\frac{(s-2)^2}{4}}-\frac{11s+2}{2}-\frac{1}{27s+11}+\frac{3s(3s+1)}{28s+10}\\
&>\frac{3s(3s+1)}{28s+10}-\frac{10s(3s+1)}{107s+10}-\frac{1}{27s+11}\\
&=\frac{3321s^4-3210s^3-6745s^2-2120s-100}{2(14s+5)(107s+10)(27s+11)}\\
&=\frac{3321(s-3)^4+36642(s-3)^3+143699(s-3)^2+229408(s-3)+115166}{2(14s+5)(107s+10)(27s+11)}\\
&>0.
\end{align*}

Therefore, $f_s(n)>0$ for all $n\geq 28s+11$. This completes the proof of Claim \ref{clm1}.
\end{proof}

\begin{claim}\label{clm2}
Let $s\geq 2$, $x\geq 28s+11$, and $1\leq h\leq s-1$. 
Then $L_s(x)>U_{s,h}(x)$.
\end{claim}

\begin{proof}
Let $f_{s,h}(x)=L_s(x)-U_{s,h}(x)$. By \eqref{eqdefR}, \eqref{eqthmdefLsn} and \eqref{eqthmdefU}, we have
\begin{equation*}
f'_{s,h}(x)=\frac{s}{2\sqrt{s(x-s)+\frac{(s-2)^2}{4}}}+\frac{1}{(x-s)^2}
-\frac{g_{s,h}'(x)}{2\sqrt{g_{s,h}(x)}}
-\frac{3h(3h+1)}{2(x-s+h-1)^2},
\end{equation*}
where $g_{s,h}(x)=(s-h)(x-s+h)+\frac{(s-h-2-R_h(x-s+h))^2}{4}$.

As verified in Appendix \ref{appemonoton}, we have $f_{s,h}'(x)>0$ for  $x\geq 28s+11$ and $1\leq h\leq s-1$. Hence $f_{s,h}(x)$ increases strictly for $x\geq 28s+11$ and $1\leq h\leq s-1$. It remains to check $f_{s,h}(x)>0$ when $x=28s+11$. The calculation in Appendix \ref{appevalue} gives $f_{s,h}(28s+11)>0$. This completes the proof of Claim \ref{clm2}.
\end{proof}

We now continue with the proof. By Lemma \ref{lemjoin}, $K_{t-1}\vee K_{\floor{(n-t+1)/2},\ceil{(n-t+1)/2}}$ is $tK_3$-free.  Let $G$ be a $tK_3$-free graph of order $n\geq 28t-17$ with maximum signless Laplacian spectral radius.  Then
\begin{equation}\label{eqthmpfqG1}
q(G)\geq q\left(K_{t-1}\vee K_{\floor{(n-t+1)/2},\ceil{(n-t+1)/2}}\right).
\end{equation}
By Claim \ref{clm1}, for $s=t-1$, we have
\begin{equation}\label{eqthmpfqG2}
q(G)>n+R_{t-1}(n).
\end{equation}

We show that $G$ is connected.  Otherwise, add one edge $uv$ between a component $C_1$ with $q(C_1)=q(G)$ and another component $C_2$. Then $G+uv$ is $tK_3$-free. Let $C=(C_1 \cup C_2)+uv$. Then $C$ is a subgraph of $G+uv$, and $C_1$ is a proper subgraph of $C$. Thus $q(G+uv)\geq q(C)>q(C_1)=q(G)$ by Lemma \ref{lemmonotone}, which contradicts the maximality of $q(G)$.

We prove by induction on $r$ that, for each $0\leq r\leq t-1$, there exists a $(t-r)K_3$-free graph $H_r$ such that
\begin{equation}\label{eqthmpfexiHr}
G=K_r\vee H_r.
\end{equation}
For $r=0$, this is trivial with $H_0=G$. 

Assume that \eqref{eqthmpfexiHr} holds for some $0\leq r\leq t-2$. We first claim that
\begin{equation}\label{qethmqHr}
 q(H_r)>|V(H_r)|+R_{t-1-r}(|V(H_r)|).
\end{equation}
When $r=0$, we have $q(H_r)>|V(H_r)|+R_{t-1-r}(|V(H_r)|)$ by \eqref{eqthmpfqG2}. When $r\geq 1$, suppose to the contrary that $q(H_r)\leq |V(H_r)|+R_{t-1-r}(|V(H_r)|)$. Clearly, $t-1\geq 2$ by the assumption $r\leq t-2$. Let the rows and columns of $Q(G)$ be ordered according to the partition $V(G)=V(K_r)\cup V(H_r)$. For any unit vector $\bm{y}\in\mathbb{R}^{|V(G)|}$, write $\bm{y}=
\begin{bmatrix}
\bm{y}_1 & \bm{y}_2
\end{bmatrix}^{\T}$ according to the partition $V(G)=V(K_r)\cup V(H_r)$. Since
\[
Q(G)=
\begin{bmatrix}
(n-1)I_r+A(K_r) & J_{r\times |V(H_r)|}\\
J_{|V(H_r)|\times r} & Q(H_r)+rI_{|V(H_r)|}
\end{bmatrix},
\]
by the Cauchy--Schwarz inequality and Rayleigh's principle, we have
\begin{align*}
\bm{y}^{\T}Q(G)\bm{y}
&= \bm{y}_1^{\T}\bigl((n-1)I_r+A(K_r)\bigr)\bm{y}_1 +2\bm{y}_1^{\T}J_{r\times |V(H_r)|}\bm{y}_2 +\bm{y}_2^{\T}\bigl(Q(H_r)+rI_{|V(H_r)|}\bigr)\bm{y}_2\\
&=(n-2)\|\bm{y}_1\|^2+\bigl(\bm{1}_r^{\T}\bm{y}_1\bigr)^2+2\bigl(\bm{1}_r^{\T}\bm{y}_1\bigr)\bigl(\bm{1}_{|V(H_r)|}^{\T}\bm{y}_2\bigr)+\bm{y}_2^{\T}\bigl(Q(H_r)+rI_{|V(H_r)|}\bigr)\bm{y}_2\\
&\leq (n+r-2)\|\bm{y}_1\|^2+2\sqrt{r|V(H_r)|}\|\bm{y}_1\|\|\bm{y}_2\|+\bigl(q(H_r)+r\bigr)\|\bm{y}_2\|^2\\
&\leq (n+r-2)\|\bm{y}_1\|^2
+2\sqrt{r|V(H_r)|}\|\bm{y}_1\|\|\bm{y}_2\|+\bigl(n+R_{t-1-r}(|V(H_r)|)\bigr)\|\bm{y}_2\|^2.
\end{align*}
Therefore, by Rayleigh's principle,
\begin{align*}
    q(G)&=\max_{\|\bm{y}\|=1}\left\{\bm{y}^{\mathsf T}Q(G)\bm{y}\right\}
\leq \max\left\{\begin{bmatrix}
\|\bm{y}_1\|&\|\bm{y}_2\|
\end{bmatrix}
\begin{bmatrix}
n+r-2 & \sqrt{r|V(H_r)|}\\
\sqrt{r|V(H_r)|} & n+R_{t-1-r}(|V(H_r)|)
\end{bmatrix}
\begin{bmatrix}
\|\bm{y}_1\|\\
\|\bm{y}_2\|
\end{bmatrix}\right\}\\
&\leq
\lambda_{\max}
\begin{bmatrix}
n+r-2 & \sqrt{r|V(H_r)|}\\
\sqrt{r|V(H_r)|} & n+R_{t-1-r}(|V(H_r)|)
\end{bmatrix}=U_{t-1,t-r-1}(n).
\end{align*}
Hence $q(G)\leq U_{t-1,t-r-1}(n)$. On the other hand, by \eqref{eqthmdefLsn}, \eqref{eqthmpfqG1} and Claim \ref{clm1}, we obtain $q(G)\geq q\left(K_{t-1}\vee K_{\floor{(n-t+1)/2},\ceil{(n-t+1)/2}}\right)\geq L_{t-1}(n)$. However, $L_{t-1}(n)>U_{t-1,t-r-1}(n)$ by Claim \ref{clm2}, a contradiction.  Thus \eqref{qethmqHr} holds.

Note that $|V(H_r)|=n-r\geq 3(t-r-1)+2$. Since $H_r$ is $(t-r)K_3$-free, by \eqref{qethmqHr} and Lemma \ref{lem1forthm}, there exists a vertex $v\in V(H_r)$ such that $H_r-\{v\}$ is $(t-r-1)K_3$-free.  If $v$ is not adjacent to every vertex of $H_r-\{v\}$, then we add all edges between $v$ and the vertices in $V(H_r)\setminus\left(N_{H_r}(v)\cup\{v\}\right)$.  The resulting graph is $K_{r+1}\vee (H_r-\{v\})$, which is still $tK_3$-free by Lemma \ref{lemjoin}.  Clearly, $G$ is a proper subgraph of $K_{r+1}\vee (H_r-\{v\})$, and thus $q(G)< q\left(K_{r+1}\vee (H_r-\{v\})\right)$ by Lemma \ref{lemmonotone}, which contradicts the maximality of $q(G)$.  Hence $N_{H_r}(v)=V(H_r)\setminus\{v\}$, and $G=K_{r+1}\vee (H_r-\{v\})$. This completes the induction, and thus \eqref{eqthmpfexiHr} holds.

When $r=t-1$, we obtain $G=K_{t-1}\vee H$, where $H=H_{t-1}$ is $K_3$-free and has order $n-t+1$. We next show that $H$ is a complete bipartite graph. Let $\bm{x}=(x_v)_{v\in V(G)}$ be a positive unit Perron vector of $Q(G)$, and choose $w\in V(H)$ such that $x_w=\max\{x_v:v\in V(H)\}$.  Let $A=N_H(w)$, and $B=V(H)\setminus A$, and $H'$ be the complete bipartite graph with parts $A$ and $B$, and let $G'=K_{t-1}\vee H'$. Note that $w\in B$.

If $A=\varnothing$, then by Lemma \ref{lem2forthm}, we have $q(G')\geq q(G)$, where $H'=(n-t+1)K_1$.  But $G'=K_{t-1}\vee (n-t+1)K_1$ is a proper subgraph of $K_{t-1}\vee K_{\floor{(n-t+1)/2},\ceil{(n-t+1)/2}}$, and thus Lemma \ref{lemmonotone} gives $q(G')<q(K_{t-1}\vee K_{\floor{(n-t+1)/2},\ceil{(n-t+1)/2}})\leq q(G)$, a contradiction.

Since $H'$ is bipartite, $H'$ is $K_3$-free, and thus $G'$ is $tK_3$-free by Lemma \ref{lemjoin}. Using Lemma \ref{lem2forthm} gives $q(G')\geq q(G)$.  By the maximality of $q(G)$, we must have $q(G')=q(G)$.  Since $A\neq \varnothing$, Lemma \ref{lem2forthm} implies $H=H'$. Therefore, $H=K_{a,b}$, where $a=|A|$ and $b=|B|$.

It remains to determine $a$ and $b$.  Let $G_{a,b}=K_{t-1}\vee K_{a,b}$, and let $A$ and $B$ be the two partite sets of $K_{a,b}$ with $|A|=a$ and $|B|=b$, where $a+b=n-t+1$. The equitable quotient matrix of $Q(G_{a,b})$ with respect to the partition $V(K_{t-1}),A,B$ is 
\begin{equation*}
\begin{bmatrix}
 n+t-3 & a & b\\
 t-1 & t+b-1 & b\\
 t-1 & a & t+a-1
\end{bmatrix},
\end{equation*}
and its characteristic polynomial is
\begin{equation*}
    p_{a,b}(x)=x^3-2(n+t-2)x^2+(n^2+(2t-4)n+2t^2-6t+4)x-2n(t-1)(t-2)-4(t-1)ab.
\end{equation*}

By Lemma \ref{lemquotient}, $q(G_{a,b})$ is the largest root of $p_{a,b}(x)$. Since $t\geq 2$, increasing $ab$ strictly decreases $p_{a,b}(x)$ for every fixed $x$.  Therefore $q(G_{a,b})$ is strictly increasing as $ab$ increases.  Consequently, $q(G_{a,b})$ attains its maximum when
\[
\{a,b\}=\left\{\floor{\frac{n-t+1}{2}},\ceil{\frac{n-t+1}{2}}\right\}.
\]
Therefore, $G=G_{a,b}=K_{t-1}\vee K_{\floor{(n-t+1)/2},\ceil{(n-t+1)/2}}$. This completes the proof.
\end{proof}

\noindent
{\bf Acknowledgments}\,

The author would like to thank Prof. Bo Ning for his helpful discussions and comments.  This work is supported by National Natural Science Foundation of China grant 12371350 and the Fundamental Research Funds for the Central Universities, Nankai University (No. 63263259).

\noindent
{\bf AI Statement}\,

The author used ChatGPT (GPT-5.5 Pro) to assist with language editing, including proofreading and grammatical correction, throughout the paper.
	
\bibliographystyle{unsrt}

\begin{thebibliography}{99}

\bibitem{BabaiGuiduli}
L. Babai, B. Guiduli,
Spectral extrema for graphs: the Zarankiewicz problem,
\emph{Electron. J. Combin.} {\bfseries 16} (2009) \#R123.

\bibitem{BondyM}
J.A. Bondy, U.S.R. Murty,
Graph Theory,
Springer, New York, 2008.

\bibitem{BrouwerH}
A.E. Brouwer, W.H. Haemers,
\emph{Spectra of Graphs},
Springer, Berlin, 2011.

\bibitem{CioabaDTWheel}
S. Cioab\u{a}, D.N. Desai, M. Tait,
The spectral radius of graphs with no odd wheels,
\emph{European J. Combin.} {\bfseries 99} (2022) 103420.

\bibitem{CvetkovicRS}
D. Cvetkovi\'{c}, P. Rowlinson, S.K. Simi\'{c},
Eigenvalue bounds for the signless Laplacian,
\emph{Publ. Inst. Math.} {\bfseries 81} (95) (2007) 11--27.

\bibitem{DasLap}
K.C. Das,
The Laplacian spectrum of a graph,
\emph{Comput. Math. Appl.} {\bfseries 48} (2004) 715--724.

\bibitem{FengYuLap}
L.H. Feng, G.H. Yu,
On three conjectures involving the signless Laplacian spectral radius of graphs,
\emph{Publ. Inst. Math.} {\bfseries 85} (99) (2009) 35--38.

\bibitem{GodsilR}
C. Godsil, G. Royle,
\emph{Algebraic Graph Theory},
Graduate Texts in Mathematics, vol. 207,
Springer-Verlag, New York, 2001.

\bibitem{HeJZKr}
B. He, Y.-L. Jin, X.-D. Zhang,
Sharp bounds for the signless Laplacian spectral radius in terms of clique number,
\emph{Linear Algebra Appl.} {\bfseries 438} (2013) 3851--3861.

\bibitem{LiuNingUnsolve}
L.L. Liu, B. Ning,
Unsolved problems in spectral graph theory,
\emph{Oper. Res. Trans.} {\bfseries 27} (4) (2023) 33--60.

\bibitem{LiuNingSparse}
L.L. Liu, B. Ning,
Spectral Tur\'{a}n-type problems on sparse spanning graphs,
\emph{Discrete Math.} {\bfseries 349} (2026) 115016.

\bibitem{LiuMiaobook}
R. Liu, L. Miao,
Spectral Tur\'{a}n problem of non-bipartite graphs: forbidden books,
\emph{European J. Combin.} {\bfseries 126} (2025) 104136.

\bibitem{Mantel}
W. Mantel,
Problem 28,
\emph{Wiskundige Opgaven} {\bfseries 10} (1907) 60--61.

\bibitem{NiWangKang23}
Z.Y. Ni, J. Wang, L.Y. Kang,
Spectral extremal graphs for disjoint cliques,
\emph{Electron. J. Combin.} {\bfseries 30} (2023) \#P1.20.

\bibitem{Niki2007}
V. Nikiforov,
Bounds on graph eigenvalues II,
\emph{Linear Algebra Appl.} {\bfseries 427} (2007) 183--189.

\bibitem{NikiZaraprob}
V. Nikiforov,
A contribution to the Zarankiewicz problem,
\emph{Linear Algebra Appl.} {\bfseries 432} (2010) 1405--1411.

\bibitem{Nikispecturan}
V. Nikiforov,
The spectral radius of graphs without paths and cycles of specified length,
\emph{Linear Algebra Appl.} {\bfseries 432} (2010) 2243--2256.

\bibitem{NikiYuanevenC}
V. Nikiforov, X.Y. Yuan,
Maxima of the $Q$-index: forbidden even cycles,
\emph{Linear Algebra Appl.} {\bfseries 471} (2015) 636--653.

\bibitem{Turan}
P. Tur\'{a}n,
On an extremal problem in graph theory,
\emph{Mat. Fiz. Lapok} {\bfseries 48} (1941) 436--452.

\bibitem{YuK2}
G. Yu,
On the maximal signless Laplacian spectral radius of graphs with given matching number,
\emph{Proc. Japan Acad. Ser. A Math. Sci.} {\bfseries 84} (2008) 163--166.

\bibitem{YuanoddC}
X.Y. Yuan,
Maxima of the $Q$-index: forbidden odd cycles,
\emph{Linear Algebra Appl.} {\bfseries 458} (2014) 207--216.

\bibitem{ZhaiLinC6}
M.Q. Zhai, H.Q. Lin,
Spectral extrema of graphs: forbidden hexagon,
\emph{Discrete Math.} {\bfseries 343} (2020) 112028.

\bibitem{ZhaiWangC4}
M.Q. Zhai, B. Wang,
Proof of a conjecture on the spectral radius of $C_4$-free graphs,
\emph{Linear Algebra Appl.} {\bfseries 437} (2012) 1641--1647.

\bibitem{ZhangLei2K3}
H.X. Zhang, Y. Lei,
Maxima of the signless Laplacian spectral radius of $2K_3$-free graphs,
\emph{Discrete Math.} {\bfseries 349} (2026) 115165.

\bibitem{ZhangWang2K3}
Y.T. Zhang, L.G. Wang,
The signless Laplacian spectral radius of $2K_3$-free graphs,
\emph{Discrete Math.} {\bfseries 347} (2024) 114075.

\bibitem{ZhangWang3K3}
Y.T. Zhang, L.G. Wang,
Maxima of the $Q$-index for $3K_3$-free graphs,
\emph{Discrete Appl. Math.} {\bfseries 358} (2024) 448--456.

\bibitem{ZhaoHLWheel}
Y.H. Zhao, X.Y. Huang, H.Q. Lin,
The maximum spectral radius of wheel-free graphs,
\emph{Discrete Math.} {\bfseries 344} (2021) 112341.

\bibitem{ZhengLiSu26}
J. Zheng, H.H. Li, L. Su,
A signless Laplacian spectral Erd\H{o}s--Stone--Simonovits theorem,
\emph{Discrete Math.} {\bfseries 349} (2026) 114665.

\bibitem{ZhengLF26}
J. Zheng, Y.T. Li, Y.-Z. Fan,
Some Tur\'{a}n-type results for the signless Laplacian spectral radius,
\emph{European J. Combin.} {\bfseries 135} (2026) 104373.

\end{thebibliography}

\appendix

\refstepcounter{section}
\section*{Appendix \thesection.\ Calculations for Claim \ref{clm2}}
\addcontentsline{toc}{section}{Appendix \thesection.\ Calculations for Claim \ref{clm2}}
Recall that $s\geq 2$, $x\geq28s+11$, $1\leq h\leq s-1$ and $f_{s,h}(x)=L_s(x)-U_{s,h}(x)$. By \eqref{eqdefR}, \eqref{eqthmdefLsn} and \eqref{eqthmdefU}, we have
\begin{align}
f_{s,h}(x)=&\sqrt{s(x-s)+\frac{(s-2)^2}{4}}-\sqrt{(s-h)(x-s+h)+\frac{\bigl(s-h-2-R_h(x-s+h)\bigr)^2}{4}} \notag\\
&-\frac{5h}{2}+\frac{3h(3h+1)}{2(x-s+h-1)}-\frac{1}{x-s}\label{appefsh}
\end{align}
and
\begin{align}
f_{s,h}'(x)=&\frac{s}{2\sqrt{s(x-s)+\frac{(s-2)^2}{4}}}
+\frac{1}{(x-s)^2}-\frac{3h(3h+1)}{2(x-s+h-1)^2}\notag\\
&-\frac{(s-h)-\dfrac{3h(3h+1)\bigl(s-h-2-R_h(x-s+h)\bigr)}{2(x-s+h-1)^2}}{
2\sqrt{(s-h)(x-s+h)+\dfrac{\bigl(s-h-2-R_h(x-s+h)\bigr)^2}{4}}}. \label{appefshdiff}
\end{align}
\subsection{Verification of $f_{s,h}'(x)>0$}\label{appemonoton}

Since
\[
\sqrt{(s-h)(x-s+h)+\frac{\bigl(s-h-2-R_h(x-s+h)\bigr)^2}{4}}\geq
\frac{\left|s-h-2-R_h(x-s+h)\right|}{2},
\]
by \eqref{appefshdiff}, we obtain
\begin{align}
f_{s,h}'(x)\geq&\frac{s}{2\sqrt{s(x-s)+\frac{(s-2)^2}{4}}}+\frac{1}{(x-s)^2}-\frac{3h(3h+1)}{(x-s+h-1)^2}\notag\\
&-\frac{s-h}{2\sqrt{(s-h)(x-s+h)+\frac{\bigl(s-h-2-R_h(x-s+h)\bigr)^2}{4}}}.\label{appefshdiffbound}
\end{align}

When $h\geq \frac{s}{2}$, we have $s-h\leq \frac{s}{2}$. Since $x-s\geq27s$, we have $s(x-s)+\frac{(s-2)^2}{4}\leq\frac{109}{108}s(x-s)$, and hence
\begin{equation*}
    \frac{s}{\sqrt{s(x-s)+\frac{(s-2)^2}{4}}}\geq\sqrt{\frac{108}{109}}\sqrt{\frac{s}{x-s}}.
\end{equation*}
Moreover,
\begin{equation*}
\frac{s-h}{\sqrt{(s-h)(x-s+h)+\frac{\bigl(s-h-2-R_h(x-s+h)\bigr)^2}{4}}}\leq
\sqrt{\frac{s-h}{x-s+h}}\leq
\frac{1}{\sqrt{2}}\sqrt{\frac{s}{x-s}}.
\end{equation*}
Since $3h(3h+1)\leq12s^2$ and
$x-s+h-1\geq x-s$, it follows from
\eqref{appefshdiffbound} and $s-h\leq \frac{s}{2}$ that
\begin{align*}
f_{s,h}'(x)&\geq\frac{1}{2}\left(\sqrt{\frac{108}{109}}-\frac{1}{\sqrt{2}}\right)\sqrt{\frac{s}{x-s}}+\frac{1}{(x-s)^2}-\frac{12s^2}{(x-s)^2}\\
&>\frac{1}{2}\sqrt{\frac{s}{x-s}}\left[\frac{7}{25}-24\left(\frac{s}{x-s}\right)^{\frac{3}{2}}\right]>0.
\end{align*}
Here we use $\sqrt{\frac{108}{109}}-\frac{1}{\sqrt{2}}
>\frac{7}{25}$, $\frac{x-s}{s}>27$, and $\frac{24}{27^{3/2}}<\frac{7}{25}$.

When $1\leq h<\frac{s}{2}$, let $A=s(x-s)+\frac{(s-2)^2}{4}$ and $B=(s-h)(x-s+h)+\frac{\bigl(s-h-2-R_h(x-s+h)\bigr)^2}{4}$. By \eqref{eqdefR}, we have $R_h(x-s+h)=6h-\frac{3h(3h+1)}{x-s+h-1}$. Since $0<R_h(x-s+h)<6h$ and $sR_h(x-s+h)+2s+2(s-h)-2s(s-h)
\geq-2s(s-h)$, a direct calculation gives
\begin{align*}
&s^2B-(s-h)^2A\\
=&s(s-h)hx+\frac{\bigl(sR_h(x-s+h)+2h\bigr)\bigl(sR_h(x-s+h)+2s+2(s-h)-2s(s-h)\bigr)}{4}\\
>&s(s-h)hx
-\frac{s(s-h)}{2}
\bigl(sR_h(x-s+h)+2h\bigr)>s(s-h)h\bigl(x-3s-1\bigr).
\end{align*}
Furthermore, $0<R_h(x-s+h)<6h$ and $h\leq s-1$ imply $-6s<s-7h-2<s-h-2-R_h(x-s+h)<s$, and thus $\left|s-h-2-R_h(x-s+h)\right|<6s$. Consequently, $A<\frac{37}{27}s(x-s)$ and $B<s(x-s)+10s^2<\frac{37}{27}s(x-s)$. Note that $h<\frac{s}{2}$. Therefore, we have
\begin{align}
\frac{s}{\sqrt{A}}-\frac{s-h}{\sqrt{B}}&=\frac{s^2B-(s-h)^2A}{\sqrt{AB}\bigl(s\sqrt{B}+(s-h)\sqrt{A}\bigr)}>\frac{s(s-h)h\bigl(x-3s-1\bigr)}{\sqrt{AB}\bigl(s\sqrt{B}+(s-h)\sqrt{A}\bigr)}\notag\\
&>\frac{s(s-h)h\bigl(x-3s-1\bigr)}{\left(\frac{37}{27}\right)^{3/2}
(2s-h)s(x-s)\sqrt{s(x-s)}}>\frac{s(s-h)h\bigl(x-3s-1\bigr)}{2\left(\frac{37}{27}\right)^{3/2}
s^2(x-s)\sqrt{s(x-s)}}\notag\\
&>\frac{s(s-h)h\bigl(x-3s-1\bigr)}{4s^2(x-s)\sqrt{s(x-s)}}>\frac{h(x-3s-1)}{8(x-s)\sqrt{s(x-s)}}. \label{appeABminus}
\end{align}
Since \(x-s\geq27s+11\), we have
\[
\frac{(x-s)(x-3s-1)}{s\sqrt{s(x-s)}}>25\sqrt{\frac{x-s}{s}}>25\sqrt{27}>96.
\]
Then
\begin{equation}\label{appeABminus2}
\frac{h(x-3s-1)}
{8(x-s)\sqrt{s(x-s)}}>\frac{12sh}{(x-s)^2}>\frac{24h^2}{(x-s)^2}\geq\frac{6h(3h+1)}{(x-s+h-1)^2},
\end{equation}
where the second inequality follows from $h<\frac{s}{2}$, and the last
inequality follows from $3h+1\leq 4h$ and $x-s+h-1\geq x-s$. Combining \eqref{appefshdiffbound}, \eqref{appeABminus} and \eqref{appeABminus2} gives
\[
f_{s,h}'(x)>\frac{1}{(x-s)^2}>0.
\]

Therefore, $f_{s,h}'(x)>0$ for every $x\geq28s+11$ and $1\leq h\leq s-1$.

\subsection{Verification of $f_{s,h}(28s+11)>0$}\label{appevalue}

By \eqref{appefsh}, we have
\begin{align}
f_{s,h}(28s+11)
={}&\sqrt{\frac{109}{4}s^2+10s+1}
-\sqrt{(s-h)(27s+h+11)
+\frac{\left(s-7h-2+
\frac{3h(3h+1)}{27s+h+10}\right)^2}{4}}
\notag\\
&-\frac{5h}{2}
+\frac{3h(3h+1)}{2(27s+h+10)}
-\frac{1}{27s+11}.
\label{appef28s11}
\end{align}

Let $X=\frac{3h(3h+1)}{27s+h+10}$, and let
\[
A=\frac{109}{4}s^2+10s+1,
\qquad
B=(s-h)(27s+h+11)+\frac{(s-7h-2+X)^2}{4},\qquad C=\frac{5h-X}{2}.
\]
Then by \eqref{appef28s11}, we have $f_{s,h}(28s+11)=\sqrt A-\sqrt B-C-\frac{1}{27s+11}$.

Since $\left(\frac{\sqrt{109}}{2}s+\frac{10}{\sqrt{109}}\right)^2
=\frac{109}{4}s^2+10s+\frac{100}{109}<A$, we have
\begin{equation}
\sqrt A>\frac{\sqrt{109}}{2}s+\frac{10}{\sqrt{109}}.
\label{appeAlower}
\end{equation}
By $h\leq s-1$, we obtain $27s+h+10-9(3h+1)=27s-26h+1\geq s+27>0$,
and thus
\begin{equation}
0<X<\frac h3.
\label{appeXbound}
\end{equation}

Let $K=\frac{\sqrt{109}}{2}s+\frac56-C$. It follows from $X>0$ and $h\leq s-1$ that
\begin{equation}\label{appeKgeq0}
    K=\frac{\sqrt{109}}{2}s+\frac56-C>
\frac{\sqrt{109}}{2}s+\frac56-\frac{5h}{2}
\geq
\frac{\sqrt{109}-5}{2}s+\frac{10}{3}>0.
\end{equation}
A direct calculation gives
\begin{align*}
K^2-B=&Xh-\frac{Xs}{2}+\frac{11X}{6}
-5h^2+\frac{59hs}{2}-\frac{h}{6}+\sqrt{109}s
\left(\frac{X}{2}-\frac{5h}{2}+\frac{5}{6}\right)
-10s-\frac{11}{36}.
\end{align*}
By \eqref{appeXbound}, $\frac {X}{2}-\frac{5h}{2}+\frac{5}{6}
<\frac {h}{6}-\frac{5h}{2}+\frac{5}{6}
=\frac{5-14h}{6}<0$. Since $\sqrt{109}<\frac{21}{2}$, it follows that
\begin{equation}\label{appeK2minusB}
    K^2-B>\frac{N}{36},
\end{equation}
where $N=X(171s+36h+66)+117hs-180h^2-6h-45s-11$.

Note that
\begin{align*}
&s(171s+36h+66)-(6s+h)(27s+h+10)\\=&9s^2+3hs-h^2+6s-10h\\=&h(11h-4)+21h(s-h)+9(s-h)^2+6(s-h)>0.
\end{align*}
Then we obtain $\frac{171s+36h+66}{27s+h+10}>6+\frac{h}{s}$. Consequently,
\[
X(171s+36h+66)=\frac{3h(3h+1)(171s+36h+66)}{27s+h+10}
>9h^2\left(6+\frac hs\right).
\]
Thus
\begin{align*}
N
&>\frac{9h(s-h)(13s-h)}s-6h-45s-11.
\end{align*}
Combining this with $h(s-h)-(s-1)=(h-1)(s-h-1)\geq 0$, $13s-h\geq12s+1$ and $h\leq s-1$ gives
\begin{align*}
N
&>\frac{9(s-1)(12s+1)}s
  -6(s-1)-45s-11\\
&=\frac{57s^2-104s-9}{s}\\
&=\frac{57(s-2)^2+124(s-2)+11}{s}>0.
\end{align*}
Therefore, $K^2>B$ by \eqref{appeK2minusB}. Since $K>0$ by \eqref{appeKgeq0}, we have $\sqrt{B}+C<
\frac{\sqrt{109}}{2}s+\frac{5}{6}$. Combining this with \eqref{appef28s11}, \eqref{appeAlower} and the facts that $27s+11\geq65$ and $\sqrt{109}<\frac{21}{2}$, we obtain
\begin{align*}
f_{s,h}(28s+11)>\frac{10}{\sqrt{109}}-\frac56
  -\frac{1}{27s+11}>\frac{20}{21}-\frac56-\frac1{65}
=\frac{283}{2730}>0.
\end{align*}
\end{document}